\documentclass[12pt,a4paper]{amsart}
\usepackage{graphicx}
\usepackage{float,epsfig}
\usepackage[russian,english]{babel}
\usepackage[psamsfonts]{amssymb}
\usepackage{amsfonts}
\usepackage{amsmath}
\usepackage{xcolor}
\usepackage{hyperref}

\newtheorem{thr}{Theorem}[section]
\newtheorem{lem}{Lemma}[section]

\theoremstyle{definition}
\newtheorem{ex}{Example}[section]
\newtheorem{rem}{Remark}[section]
\def\Re{\mathop{\rm Re}\nolimits}
\def\Im{\mathop{\rm Im}\nolimits}

\begin{document}

\title[]{ONE-PARAMETER FAMILIES OF CONFORMAL MAPPINGS OF THE HALF-PLANE ONTO POLYGONAL DOMAINS WITH SEVERAL SLITS}%\thanks{The work of the first author was supported by the
%Russian Foundation for Basic Research and the Government of the
%Republic of Tatarstan, grant No~18-41-160003; the second author
%was supported by the Russian Foundation for Basic Research, grant
%No~17-01-00282. The third author expresses his thanks to the Kazan
%Regional Scientific and Educational Mathematical Center for a
%support during his stay at Kazan Federal University in
%October-November 2018.}

\author[A.~Posadskii]{A.~Posadskii}
\address{Lebedev Physical Institute,
Moscow, 119991; Saint Petersburg University, 7/9 Universitetskaya nab., St. Petersburg, 199034 Russia}
         \email[]{posadskij.af@phystech.edu}
\author[S.~Nasyrov]{S.~Nasyrov}
\address{Kazan Federal University, Kazan, 420008; Saint Petersburg University, 7/9 Universitetskaya nab., St. Petersburg, 199034 Russia}
 \email[]{semen.nasyrov@yandex.ru}

\maketitle

\begin{abstract}
Among various methods of finding accessory parameters in the Schwarz-Christoffel integrals, Kufarev's method, based on the Loewner differential equation, plays an important role. It is used for describing one-parameter families of functions that conformally map a canonical domain onto a polygon with a slit the endpoint of which moves along a polygonal line starting from a boundary point. We present a modification of Kufarev's method for the case of several slits, the lengths of which have depend of each other in a certain way. We justify the method and find a system of ODEs describing the dynamics of accessory parameters. We also present the results of numerical calculations which confirm the efficiency of our method.\\


\end{abstract}

%%%%%%%%%%%%%%%%%%%%%%%%%%%
%%%%%%%%%%%%%%%%%%%%%%%%%%%
%%%%%%%%%%%%%%%%%%%%%%%%%%%
\section*{Introduction}\label{intr}
%%%%%%%%%%%%%%%%%%%%%%%%%%%
%%%%%%%%%%%%%%%%%%%%%%%%%%%
%%%%%%%%%%%%%%%%%%%%%%%%%%%

As it is known, the integral formula giving a conformal mapping of a canonical domain onto a given polygonal domain was independently derived by E.~Christoffel \cite{Christoffel} and H.~Schwarz \cite{Schwarz}.  The Schwarz-Christoffel integral contains unknown (accessory) parameters, and in general case the problem to find them is very difficult (see, e.g. \cite{Driscoll_Tr}).

Many approaches have developed to solve this problem, including analytical methods suggested by M.~A.~Lavrentiev \cite{Lavrentiev1,Lavrentiev2},  A.~Weinstein \cite{Weinstein1,Weinstein2}, S.~Bergman \cite{Bergmann1,Bergmann2} and others. With the development of computing power, numerical methods are being actively developed; some of them are presented in the works of V.~V.~Sobolev \cite{Sobolev1,Sobolev2}, L.~N.~Trefesen \cite{Trefethen}, R.~T.~Davis \cite{Davis} and others. Later these methods were  implemented in software packages (see, e.g., the Driscoll MATLAB toolbox~\cite{Driscoll}). We also note  a new interesting method, suggested by S.~I.~Bezrodnykh \cite{bezr}; it is based on analytic continuation of the Lauricella function.

%\textbf{PAPERS BY BEZRODNIKH}

The famous Loewner parametric method \cite{Goluzin,Alexandrov} has influenced the process of solving the accessory parameters problem. The main idea of the method is to consider families of univalent functions depending on a real parameter and mapping some canonical domain (such as the unit disk or the upper half-plane) onto a family of domains with a growing slit. If the trajectory of the slit is a polygonal line and the endpoint of the slit moves along one of its line segment, then the family satisfies some differential equation. Using this approach, P.~P.~Kufarev (see, e.g. \cite{Kufarev}, \cite[Ch.VI]{Alexandrov}) reduced the problem of finding the accessory parameters  to integrating several Cauchy problems for systems of ordinary differential equations. In this method, the process of finding the accessory parameters consists of several steps. At every step, one of such systems is solved and its initial conditions are determined from the results obtained on the previous step. The values obtained at the last step, are desirable ones. Some numerical results based on the Kufarev's method can be found in \cite{Chistyakov} and \cite{Hopkins}.

Later, V.~Ya.~Gutlyansky and A.~O.~Zaidan \cite{Gutlyansky} proposed a modification of Kufarev's method. They consider conformal mappings in the upper half-plane $\mathbb{H}^+$ and cut the needed polygon $P$ out of another polygon $P'$, not a canonical domain. Geometrically, this means that at first a rectilinear slit grows from a boundary point of a polygon $P'$, then, at the second step, the slit changes its direction and grows from the endpoint of the slit obtained at the first step and so on.  At the last step,  the endpoint of the slit reaches  the boundary of $P'$ (see Fig.~\ref{fig:1 1.P-P}.). The whole trajectory $L$ of the endpoint of the slit subdivides $P'$ into two parts, every of which is a polygonal domain; they are kernels of the family of domains with growing slit. If one of the kernels is the needed polygonal domain $P$ and the normalization of functions, mapping conformally $\mathbb{H}^+$ onto domains with slits along subarcs of $L$, is fixed appropriately, then, according to the Caratheodory kernel convergence theorem (see, e.g., \cite[ch.~II]{Goluzin}, the function corresponding to the limiting (final) values of the accessory parameters maps the upper half-plane onto $P$. Thus, we cut out, like scissors, the polygonal domain  $P$ with vertices $A_1$, $A_2,\ldots,A_{n}$ from the original polygon $P'$ with vertices $A'_1$, $A'_2,\ldots,A'_{n'}$.  On Fig.~1 we illustrate the situation in a case where; here $P'$ and $P$ have common vertices $A'_{j}=A_j$, $j\in \left\{1,\ldots,l,k+3,\ldots,n\right\}$).

\begin{figure}[h]
\vspace{-2ex}
	\includegraphics[scale=0.056]{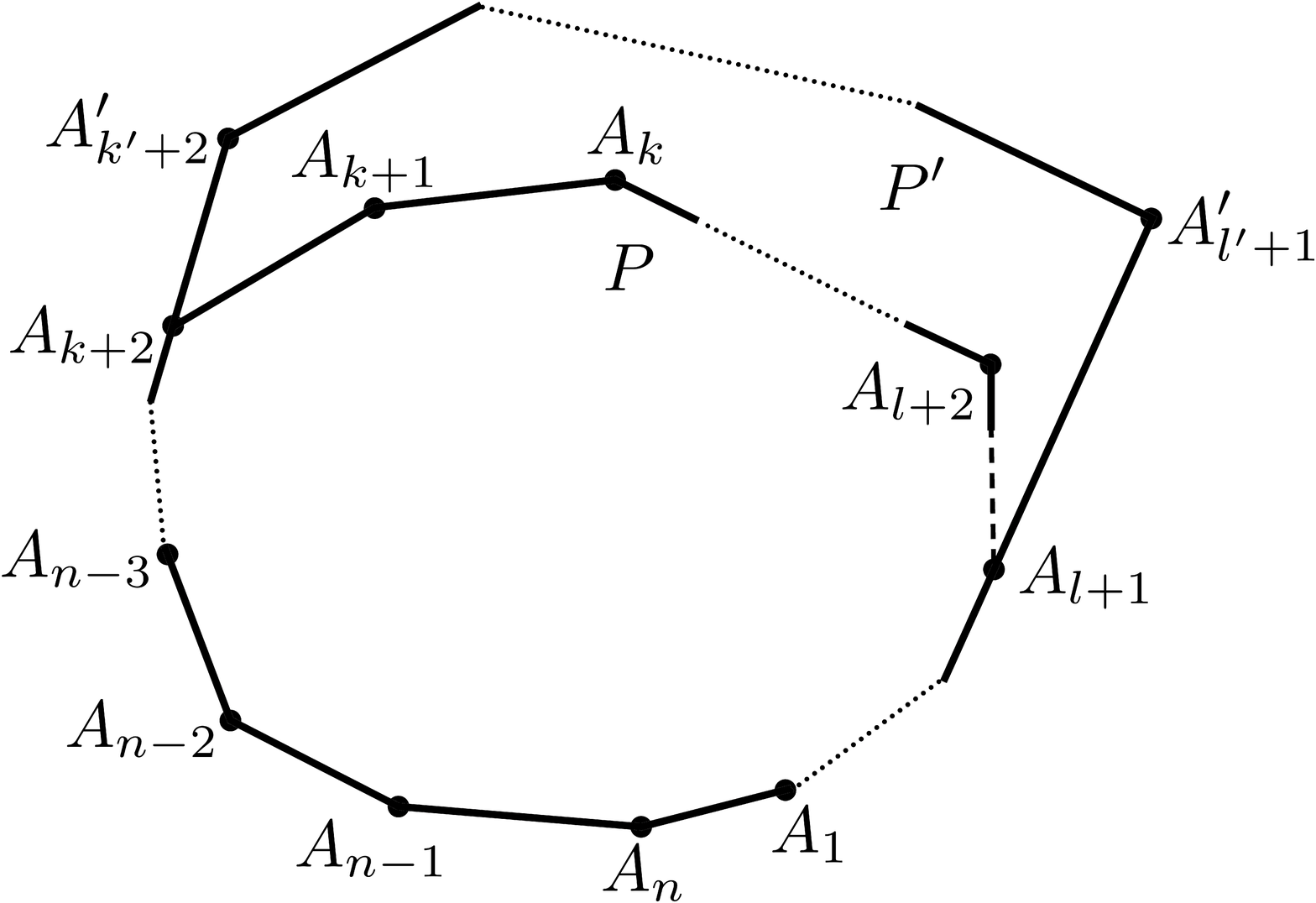}
        \centering
	\caption{Polygon $P'$ with a growing  slit along a part of the boundary of  the polygon $P$.} % from the point $A_{k+2}$. On figure $A'_{i'}=A_i,$ at $i'=i\in %\left\{1,\ldots,l,k+3,\ldots,n\right\}$.}
	\label{fig:1 1.P-P}
\end{figure}

Kufarev's method showed its efficiency and there arises an interest in its generalization. Thus, using the parametric method, it is possible to find approximate conformal mappings of canonical domains onto polycircular arc domains (see, for example, \cite{Alexandrov}). Recently,  I.~A.~Kolesnikov  \cite{Kolesnikov1, Kolesnikov2} proposed an improvement of this approach. Also he used the  parametric method to find accessory parameters for conformal mappings of polygonal domains with a countable set of vertices and the translation symmetry \cite{Kolesnikov3}. Moreover I.~A.~Kolesnikov \cite{Kolesnikov4} developed a more general method that allows to carry out parallel translation of the sides of a polygon. L.~Yu.~Nizamieva \cite{Nizamieva1} suggested an approximate method for finding accessory parameters in the Schwarz-Christoffel integrals,  based on Kufarev's method and the apparatus of Hilbert's boundary value problems. In addition, she develop this approach for solving some boundary value problems with a free boundary \cite{Nizamieva2}. A generalization for polygonal domains on Riemann surfaces  was given by N.~N.~Nakipov and S.~R.~Nasyrov \cite{Nakipov}.

In this paper, we suggest a modification %Kufarev's method
of the results obtained by V.~Ya.~Gutlyansky and A.~O.~Zaidan \cite{Gutlyansky} for the case where several slits in an initial polygon grow simultaneously. The advantage of the proposed approach is that it becomes possible to reduce significantly the number of successive steps, in comparison to \cite{Gutlyansky}. When solving the problem, we can perform calculations in fewer steps. For example, for a convex polygon we can make it in one step.

Now we will describe one of the possible variants. Let we have a polygonal domain $P'$ with vertices $A'_1$, $A'_2,\ldots,A'_{n'}$ and we want to cut out of $P'$ a convex polygonal domain  $P$ with vertices $A_1$, $A_2,\ldots,A_n$ from $P'$. Assume that we obtain $P$ from $P'$ drawing the slit along the polygonal line $A_{l+1}A_{l+2}\ldots A_{k+2}$ (Fig.~\ref{OdinEtap}). Then, for every $j$, $l+1\le j\le k+1$, we draw the ray $L_j$ originated at $A_j$ and passing through the next vertex $A_{j+1}$. Let the ray $L_j$ intersect the boundary of the polygon $P'$ at some point $B_j$ ($B_{k+1}=A_{k+2}$). Then we simultaneously release slits from the points $B_j$, $l+1\le j\le k+1$, which endpoints move along the rays $L_j$ to the points $A_{j}$. The growth velocities of slits can be chosen so that the endpoints of the $j$th slits reach simultaneously the points $A_{j}$, $l+1\le j\le k+1$. As a result, we see that, unlike the classic Kufarev's method, we reduced  the number of steps to one. The only thing is that for carrying out such an algorithm it is necessary to solve a more complicated system of ordinary differential equations. However, the system does not essentially differ from the systems of ODEs used in the case of one slit. Thus, we see that the developed method can be used for a more simple and fast search for the accessory parameters.

\begin{figure}[h]
        \vspace{1.5ex}
	\includegraphics[scale=0.06]{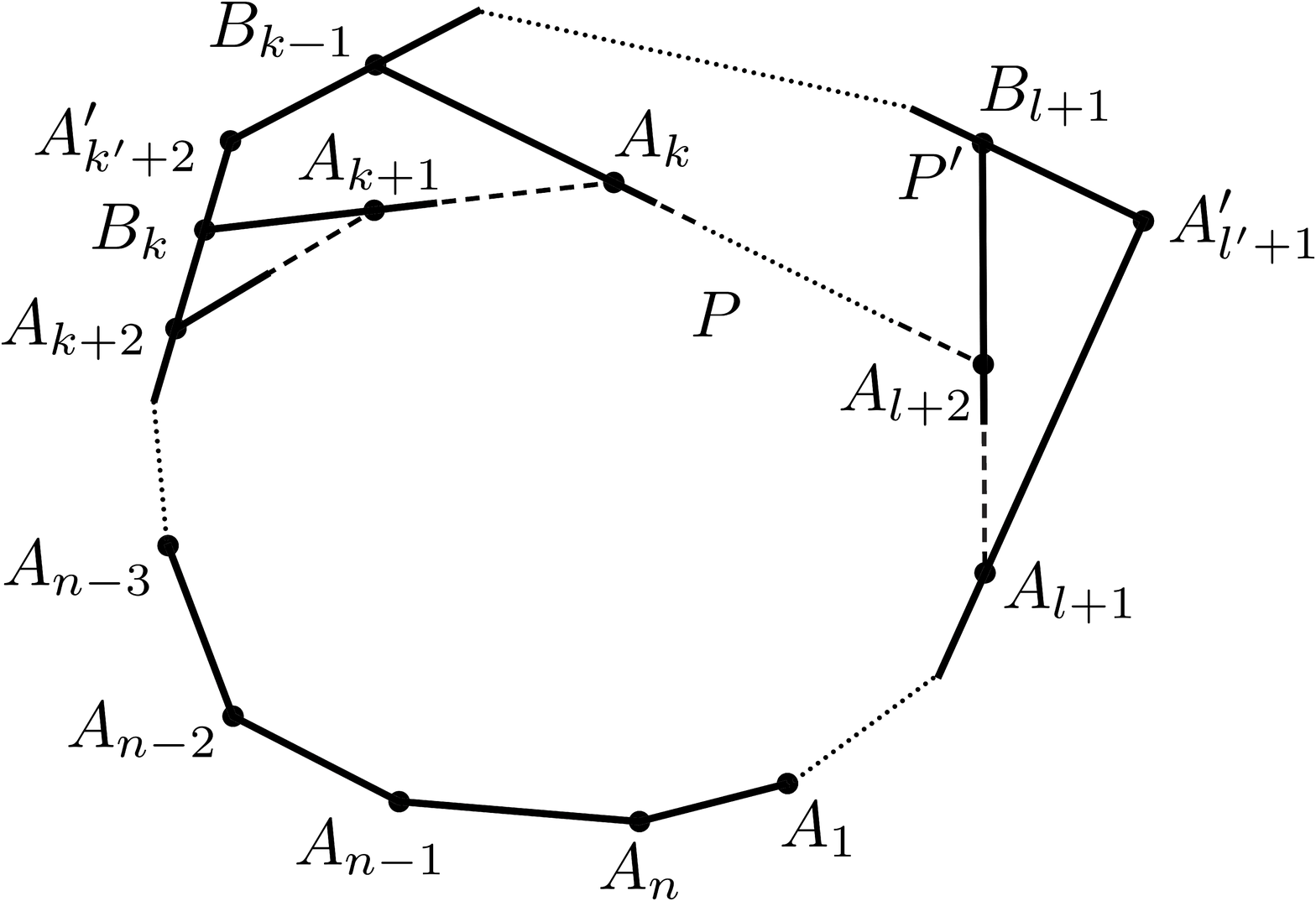}
        \centering
	\caption{Cutting a convex polygonal domain $P$ out of $P'$ in one stage. On the figure, $A'_{i'}=A_i$ for $i'=i\in \left\{1,\ldots,l,k+3,\ldots,n\right\}$.}
	\label{OdinEtap}
        \vspace{-1.5ex}
\end{figure}

To justify the method, it is important to choose an appropriate parameter with respect to which the family of mappings is differentiable. We consider the Schwarz-Christoffel integral of a sufficiently arbitrary form giving a conformal mapping onto a polygonal domain with fixed angles and investigate the dependence between the lengths of the sides of polygon and the accessory parameters. We show that the accessory parameters depends smoothly on  the lengths of the sides (Actually, this fact was established by A.~Weinstein \cite{Weinstein1}, \cite{Weinstein2}). In the case of a fixed polygonal domain with growing slits, such that the lengths are connected with each other smoothly, this fact allows us to take, as a suitable parameter, the length of one of these slits.

%%%%%%%%%%%%%%%%%%%%%%%%%%%
%%%%%%%%%%%%%%%%%%%%%%%%%%%
%%%%%%%%%%%%%%%%%%%%%%%%%%%
\section{One-parameter families of conformal mappings onto polygons with several slits}\label{sec1}
%%%%%%%%%%%%%%%%%%%%%%%%%%%
%%%%%%%%%%%%%%%%%%%%%%%%%%%
%%%%%%%%%%%%%%%%%%%%%%%%%%%

Let $D_{n}$ be a simply connected domain in $\mathbb{C}$, whose boundary is a (not necessarily bounded) $n$-gon, $n\geq 3$, with vertices $A_{1},\ldots,A_{n}$ and interior  angles $\pi\alpha_{k}$,  $\left|\alpha_{k}\right|\leq2$, $1\le k\le n$ (the value of $\alpha_k$ is negative if $A_k=\infty$). Then the conformal mapping of the upper half-plane $\mathbb{H}^{+}$ onto $D_{n}$ exists and can be represented by the Schwarz-Christoffel formula:
\begin{equation}
    f\left(z\right)=c\int_{0}^{z}\prod_{k=1}^{n}\left(\zeta-a_{k}\right)^{\alpha_{k}-1}d\zeta+c_1,
    \label{2.1}
\end{equation}
where $a_{1},\ldots,a_{n}$ are the preimages of the vertices $A_{1},\ldots,A_{n}$, $c\neq 0$ and $c_{1}$  are some complex numbers. The numbers $a_k$, $1\le k\le n$, $c$ and $c_{1}$ are not known beforehand, they are called {\it accessory parameters}. Moreover, due to the properties of conformal mappings, we can fix any three of the parameters $a_k$. Then the remaining accessory parameters  and, therefore,  the conformal mapping are defined in a unique way. Further we will assume that $a_{n-2}=0$, $a_{n-1}=1$, $a_{n}=\infty$, and  the vertices $A_{n-2}$, $A_{n-1}$, $A_{n}$ are at finite points of the plane; this does not essentially restrict the generality. In this case, $-\infty<a_1<a_2<\ldots<a_{n-3}<0$ and (\ref{2.1}) takes the form
\begin{equation*}
    f\left(z\right)=c\int_{0}^{z}\prod_{k=1}^{n-1}\left(\zeta-a_{k}\right)^{\alpha_{k}-1}d\zeta+A_{n-2}.
    \label{2.2}
\end{equation*}

Let us fix $m$ different points $B_{1}$, $B_{2},\ldots,B_{m}$ on $\partial D_{n}$, not lying on the segments $A_{n-2}A_{n-1}$ and $A_{n-1}A_{n}$, and release from them inside  $D_{n}$ disjoint  rectilinear slits of growing lengths depending on a real parameter~$t\in[0,T]$. Denote by $\Lambda_{1}(t),\ldots,\Lambda_{m}(t)$ the endpoints of the slits. The domain $D_{n}$ with slits, corresponding to the value of the parameter $t$, will be denoted by $D_{n}(t)$. We also assume that at $t=0$ the lengths of all slits are equal to zero, and for $t\ge0$ they are strictly increasing. Since the domain $D_{n}(t)$ is also polygonal, the conformal mapping $f\left(z,t\right):\mathbb{H}^{+}\rightarrow D_{n}(t)$, satisfying the normalization conditions $f\left(0,t\right)=A_{n-2},f\left(1,t\right)=A_{n-1},f\left(\infty,t\right)=A_{n}$, can be also  represented as a Schwarz-Christoffel integral.

For every $i$, $1\le i\le m$, the length of the $i$-th slit grows, and at the initial moment $t=0$ the point $B_i$ bifurcates. Therefore, for $t>0$ there are two boundary elements (prime ends in the sense of Caratheodory) in $D_{n}(t)$  supported at the point $B_i$; they lie on the different sides of the slit. Denote by $a_{i1}(t)$ and $a_{i2}(t)$ their preimages under the conformal mapping of the upper half-plane onto $D_{n}(t)$, and let $\pi\alpha_{i1}$ and $\pi\alpha_{i2}$ be the angles of $D_{n}(t)$ at these points. We numerate $a_{ij}$ so that $a_{i1}(t)<a_{i2}(t)$, $1\le i \le m$. {Denote by $\lambda_{i}(t)$ the preimage of the endpoint of the $i$-th slit; it is located on the segment %with end points $a_{i1}(t)$ and $a_{i2}(t)$
$\left[a_{i1}(t),a_{i2}(t)\right]$.} Let also $\sigma_{k}=\alpha_{k}-1$, $\sigma_{ij}=\alpha_{ij}-1$.

Now we will write the Schwarz-Christoffel integral mapping the upper half-plane onto $D_{n}(t)$. Here we distinguish two cases.

1) If none of the slits leaves a vertex of the polygon $\partial D_{n}$, then for all $i$ we have $\alpha_{i1}+\alpha_{i2}=1$, i.e. $\sigma_{i1}+\sigma_{i2}=-1$, and
\begin{equation}
    f\left(z,t\right)=c(t)\int_{0}^{z}\prod_{l=1}^{m}\left(\zeta-\lambda_{l}(t)\right)\prod_{i=1}^{m}\prod_{j=1}^{2}\left(\zeta-a_{ij}(t)\right)^{\sigma_{ij}}
    \prod_{k=1}^{n-1}\left(\zeta-a_{k}(t)\right)^{\sigma_{k}}d\zeta+A_{n-2}.
    \label{2.3}
\end{equation}

2) If one of the initial points of the slits coincides with a vertex of the initial polygon, for example, $B_{l}=A_{k}$, then we omit the factor in the formula $\left(\ref{2.3}\right)$ with  $k=l$; we do the same for a larger number of slits growing from vertices.

In what follows, for greater clarity, we will consider only case 1), unless otherwise stated, although considering the case 2) is not  difficult.

%%%%%%%%%%%%%%%%%%%%%%%%%%%
%%%%%%%%%%%%%%%%%%%%%%%%%%%
%%%%%%%%%%%%%%%%%%%%%%%%%%%
\section{Derivation of the Loewner equation}\label{sec2}
%%%%%%%%%%%%%%%%%%%%%%%%%%%
%%%%%%%%%%%%%%%%%%%%%%%%%%%
%%%%%%%%%%%%%%%%%%%%%%%%%%%

Information about the dynamics of the mapping \eqref{2.3} with a change in the parameter $t$ is given by the Loewner equation. Moreover, the equation allows us to find a system of ordinary differential equations to determine the accessory parameters. For the case of mapping of the unit disk, the corresponding equations can be found, for example, in the monograph~\cite{Alexandrov}.

Here we consider the case of several slits. When deriving the Loewner equation, we will follow the scheme, proposed in \cite{Gutlyansky} where the case of a single slit was considered.

We start by proving a consequence of the Schwarz lemma.

\begin{lem}
Let $w(z)$ be a conformal mapping of $\mathbb{H}^{+}$ onto $\mathbb{H}^{+}$ with $m$ slits along disjoint analytic arcs $\Gamma_{1}$, $\Gamma_{2},\ldots,\Gamma_{m}$,  growing from the points $\lambda_{1},\lambda_{2},\ldots,\lambda_{m}$ on $\partial H$ such that $-\infty<\lambda_{1}<\lambda_{2}<\ldots<\lambda_{m}<0$ and $w(0)=0$, $w(1)=1$, $w(\infty)=\infty$. Then $w(z)$ is holomorphic at $z=1$ and $\left.\frac{dw}{dz}\right|_{z=1}\leq1$. If $\left.\frac{dw}{dz}\right|_{z=1}=1$, then $w\left(z,t\right)\equiv z$.
\label{lemm1}
\end{lem}

\begin{proof}
There are $m$ segments $\left[\alpha_{1},\beta_{1}\right],\left[\alpha_{2},\beta_{2}\right],\ldots,\left[\alpha_{m},\beta_{m}\right]\in\mathbb{R}^{-}$ which the conformal mapping $w\left(z\right)$  maps to the slits along the arcs $\Gamma_{1}$, $\Gamma_{2},\ldots,\Gamma_{m}$. We have $\alpha_{i}<\beta_{i}<\alpha_{i+1}<\beta_{i+1}$, $1\le i\le m-1$.

By the Riemann-Schwarz symmetry principle, $w\left(z\right)$ can be  analytically continued  into the lower half-plane through $\mathbb{R}^+$. In particular, $w(z)$ is holomorphic at the point $z=1$ and, using the geometric meaning of the derivative argument, we conclude that $\Im\left.\frac{dw}{dz}\right|_{z=1}=0$. The function $\phi\left(\zeta\right)=(\zeta+1)^2(\zeta-1)^{-2}$ provides a conformal mapping of the unit circle $\left|\zeta\right|<1$ onto a plane with a slit along $\mathbb{R^{-}}$.

Applying the Schwarz lemma to the  holomorphic in $|\zeta|<1$ function $\Phi\left(\zeta\right)=\left(\phi^{-1}\circ w\circ\phi\right)\left(\zeta\right)$ we obtain the required assertion.
\end{proof}

 Now we will return to consideration of the family of functions $f(z,t)$, defined by  \eqref{2.3}. We introduce the function $w\left(z,t\right)=f^{-1}\left(f\left(z,T\right),t\right)$, where $0\leq t\leq T$, conformally mapping $\mathbb{H}^{+}$ onto $\mathbb{H}^{+}$ with $m$ slits  and satisfying $w(0,t)=0$, $w(1,t)=1$, $w(\infty,t)=\infty$. The points on $\partial H$ from which the slits grow are located on $\mathbb{R^{-}}$. By the Caratheodory kernel convergence theorem, $w(z,\,\cdot\,)$ is continuous on $\left[0,T\right]$ for a fixed $z\in \mathbb{H}^{+}$. By virtue of Lemma~\ref{lemm1} and the Weierstrass theorem on the convergence of a sequence of holomorphic functions,
 \begin{equation}\label{q1wz}
 \frac{dw}{dz}\left(1,t\right)=q(t)
 \end{equation}
 is a non-negative continuous strictly increasing function of the parameter $t$. Indeed, if $0\leq s< t\leq T$ then, by Lemma~\ref{lemm1}, applied to the function $w^{-1}\left(w\left(z,s\right),t\right)$, we conclude that $\frac{dw}{dz}\left(1,s\right)\left(\frac{dw}{dz}\left(1,t\right)\right)^{-1}<1$.

For convenience,  we will reparametrize our family, taking the value $\tau=\ln q(t)$ as a new parameter. In what follows, for simplicity of notation, we will assume that $q(t)=\exp(t)$; such a parametrization will be called canonical.

Now consider the function $h(z,t,s)=f^{-1}(f(z,t),s)$, $0\le s\le t\le T$, conformally mapping $\mathbb{H}^{+}$ onto $\mathbb{H}^{+}$ with $m$ slits; it keeps the points $0$, $1$, and $\infty$. According to the Schwarz formula for the half-plane, we obtain:
\begin{equation}\label{q}
    h\left(z,t,s\right)=z+\frac{z\left(z-1\right)}{\pi}\sum_{i=1}^{m}I_{i}\left(z,t,s\right),
\end{equation}
where
\[
I_{i}\left(z,t,s\right)=\int_{\alpha_{i}\left(s\right)}^{\beta_{i}\left(s\right)}\frac{\Im h\left(x,t,s\right)}{\left(x-z\right)\left(x-1\right)x}dx.
\]
Since $h\left(w\left(z,t\right),t,s\right)=w\left(z,s\right)$, this implies
\begin{equation}\label{st}
	w\left(z,s\right)-w\left(z,t\right)=\frac{w\left(z,t\right)\left(w\left(z,t\right)-1\right)}{\pi}\sum_{i=1}^{m}I_{i}\left(w\left(z,t\right),t,s\right).
\end{equation}
Going to the limit, as $z\rightarrow1$, from \eqref{q} and the definition of $h(z,t,s)$ we obtain
\begin{equation*}
    \exp\left(s-t\right)=\frac{dh}{dz}\left(1,t,s\right)=1+\frac{1}{\pi}\sum_{i=1}^{m}J_{i}\left(t,s\right),
\end{equation*}
where
\begin{equation*}
   J_{i}\left(t,s\right)= I_{i}\left(1,t,s\right)=\int_{\alpha_{i}\left(s\right)}^{\beta_{i}\left(s\right)}\frac{\Im h\left(x,t,s\right)}{\left(x-1\right)^{2}x}\,dx<0.
\end{equation*}
Therefore, we have
\begin{equation}\label{st}
s-t\sim  \exp\left(s-t\right)-1 =\frac{1}{\pi}\sum_{i=1}^{m}J_{i}\left(t,s\right), \quad s\to t.
\end{equation}

By the mean value theorem, we have
\begin{equation*}
    I_{i}\left(w\left(z,t\right),t,s\right)=\int_{\alpha_{i}\left(s\right)}^{\beta_{i}\left(s\right)}\frac{\Im h\left(x,t,s\right)}{\left(x-1\right)^{2}x}dx\cdot K_{i}=J_{i}\left(t,s\right)\cdot K_{i},
\end{equation*}
where
\begin{equation*}
    K_{i}=\Re\frac{x_{i,1}-1}{x_{i,1}-w\left(z,t\right)}+i\Im\frac{x_{i,2}-1}{x_{i,2}-w\left(z,t\right)}, \quad x_{i,1},x_{i,2}\in\left[\alpha_{i}\left(s\right),\beta_{i}\left(s\right)\right],
\end{equation*}
therefore, from \eqref{st} we obtain
\begin{equation}\label{partw}
    \frac{\partial w}{\partial t}=\lim_{s\rightarrow t}\frac{w\left(z,t\right)\left(w\left(z,t\right)-1\right)}{\pi}\sum_{i=1}^{m}\left(J_{i}\left(t,s\right)\cdot K_{i}\right)\left(\frac{1}{\pi}\sum_{k=1}^{m}J_{k}\left(t,s\right)\right)^{-1}.
\end{equation}

Below we will show that if the lengths of the slits smoothly depend on each other, for every $i$,  $1\le i\le m$, there exist the limit
\begin{equation*}
    C_{i}(t)=\lim_{s\rightarrow t}\left(J_{i}\left(t,s\right)\left(\sum_{k=1}^{m}J_{k}\left(t,s\right)\right)^{-1}\right).
    \label{lim}
\end{equation*}

We note that $C_{i}(t)\ge0$ as the limit of the ratio of two negative numbers. In addition, there is a relationship between these coefficients:
\begin{equation*}
    \sum_{i=1}^{m}C_{i}(t)=1.
\end{equation*}

Then, taking into account the introduced notation, we obtain
\begin{equation*}
    \frac{\partial w}{\partial t}=w\left(z,t\right)\left(w\left(z,t\right)-1\right)\sum_{i=1}^{m}C_{i}(t)\frac{\lambda_{i}(t)-1}{\lambda_{i}(t)-w\left(z,t\right)}.
\end{equation*}
Recalling that  $w\left(z,t\right)=f^{-1}\left(f\left(z,T\right),t\right)$, we get
\begin{equation}
    \frac{\partial f}{\partial t}=-\frac{\partial f}{\partial z}z\left(z-1\right)\sum_{i=1}^{m}C_{i}(t)\frac{\lambda_{i}(t)-1}{\lambda_{i}(t)-z}.
    \label{2.16}
\end{equation}

At last, we should note that, in fact, we consider the limit if   $s$ approaches $t$ from the left; the limit from the right is studied in a similar way.\medskip

\begin{rem}\label{r1}  If we did not use the change to the canonical parameter, we would get the equation similar to \eqref{2.16} but with the  additional multiplier $d(\ln q(t))/dt$  in the right-hand side; under some natural conditions, we will show the smoothness of the function $q(t)$ in Subsection~6.2. Therefore, the form of \eqref{2.16} do not change, since we can multiply  every $C_i(t)$ by this expressions and denote the result of multiplication, say, by $\widetilde{C}_i(t)$.
\end{rem}

\begin{rem}\label{r2} We can consider similar families of conformal mappings with normalization such that $f(0,t)$, $f(1,t)$  and $f(\infty,t)$ do not depend on $t$ but are not necessarily angular points of the corresponding polygonal domains. Then, as it is easy to see, the form of the derived equation does not change.
\end{rem}

%%%%%%%%%%%%%%%%%%%%%%%%%%%
%%%%%%%%%%%%%%%%%%%%%%%%%%%
%%%%%%%%%%%%%%%%%%%%%%%%%%%
\section{ODE system for accessory parameters}\label{fam}
%%%%%%%%%%%%%%%%%%%%%%%%%%%
%%%%%%%%%%%%%%%%%%%%%%%%%%%
%%%%%%%%%%%%%%%%%%%%%%%%%%%
Now we will derive a system of equations for determining the accessory parameters.

\begin{thr}
The accessory parameters satisfy on $[0,T]$ the following system of differential equations:
\[
    -\frac{d\lambda_{p}}{dt}=\lambda_{p}(t)\left(\lambda_{p}(t)-1\right)\left(\sum_{l=1,l\neq p}^{m}C_{l}(t)\frac{\lambda_{l}(t)-1}{\lambda_{p}(t)-\lambda_{l}(t)}+C_{p}(t)\left(\lambda_{p}(t)-1\right)\right.
\]
\[
    \left.\times\sum_{l=1,l\neq p}^{m}\frac{1}{\lambda_{p}(t)-\lambda_{l}(t)}\right)+C_{p}(t)\left(2\lambda_{p}(t)-1\right)\left(\lambda_{p}(t)-1\right)+C_{p}(t)\lambda_{p}(t)
\]
\begin{equation}
    \times\left(\lambda_{p}(t)-1\right)^{2}\left(\sum_{i=1}^{m}\sum_{j=1}^{2}\frac{\sigma_{ij}}{\lambda_{p}(t)-a_{ij}(t)}+\sum_{k=1}^{n-1}\frac{\sigma_{k}}{\lambda_{p}(t)-a_{k}(t)}\right),\ 1\leq p\leq m,
    \label{2.17}
\end{equation}
\begin{equation}
    \frac{d a_{l}}{dt}=-a_{l}(t)\left(a_{l}(t)-1\right)\sum_{i=1}^{m}C_{i}(t)\frac{\lambda_{i}(t)-1}{a_{l}(t)-\lambda_{i}(t)},\ 1\leq l\leq n-3,
    \label{2.18}
\end{equation}
\begin{equation}
   \frac{d a_{ij}}{dt}=-a_{ij}(t)\left(a_{ij}(t)-1\right)\sum_{l=1}^{m}C_{l}(t)\frac{\lambda_{l}(t)-1}{a_{ij}(t)-\lambda_{l}(t)},\ 1\leq i\leq m,\ 1\leq j\leq2,
   \label{2.19}
\end{equation}
\begin{equation}
    \frac{1}{c(t)}\frac{dc}{dt}=-\alpha_{n}\sum_{i=1}^{m}C_{i}(t)\left(\lambda_{i}(t)-1\right).
    \label{2.20}
\end{equation}\\
\end{thr}

\begin{proof}
Consider  the  function $\phi\left(z,t\right)=\ln\frac{\partial f}{\partial z}$. Because of  \eqref{2.16}, it satisfies the differential equation
\[
    -\frac{\partial\phi}{\partial t}=\frac{\partial\phi}{\partial z}z\left(z-1\right)\sum_{i=1}^{m}C_{i}(t)\frac{\lambda_{i}(t)-1}{\lambda_{i}(t)-z}
\]
\begin{equation}
    +\sum_{i=1}^{m}C_{i}(t)\frac{\lambda_{i}(t)\left(\lambda_{i}(t)-1\right)^{2}}{\left(\lambda_{i}(t)-z\right)^{2}}-\sum_{i=1}^{m}C_{i}(t)\left(\lambda_{i}(t)-1\right).
    \label{2.21}
\end{equation}
From \eqref{2.3} it follows that
\begin{equation*}
    \phi\left(z,t\right)=\ln c(t)+\sum_{l=1}^{m}\ln\left(z-\lambda_{l}(t)\right)
    +\sum_{i=1}^{m}\sum_{j=1}^{2}\sigma_{ij}\ln\left(z-a_{ij}(t)\right)+\sum_{k=1}^{n-1}\sigma_{k}\ln\left(z-a_{k}(t)\right).
\end{equation*}
Therefore,
\begin{equation}
    -\frac{\partial\phi}{\partial t}=-\frac{1}{c(t)}\frac{dc}{dt}+\sum_{l=1}^{m}\frac{1}{z-\lambda_{l}(t)}\frac{d\lambda_{l}}{dt}
    +\sum_{i=1}^{m}\sum_{j=1}^{2}\frac{\sigma_{ij}}{z-a_{ij}(t)}\frac{d a_{ij}}{dt}+\sum_{k=1}^{n-1}\frac{\sigma_{k}}{z-a_{k}(t)}\frac{d a_{k}}{dt},
    \label{2.23}
\end{equation}
\begin{equation}
    \frac{\partial\phi}{\partial z}=\sum_{l=1}^{m}\frac{1}{z-\lambda_{l}(t)}+\sum_{i=1}^{m}\sum_{j=1}^{2}\frac{\sigma_{ij}}{z-a_{ij}(t)}+\sum_{k=1}^{n-1}\frac{\sigma_{k}}{z-a_{k}(t)}.
    \label{2.24}
\end{equation}

Substituting (\ref{2.23}) and (\ref{2.24}) into  (\ref{2.21}), we get
\[
    -\frac{1}{c(t)}\frac{dc}{dt}+\sum_{l=1}^{m}\frac{1}{z-\lambda_{l}(t)}\frac{d\lambda_{l}}{dt}+\sum_{i=1}^{m}\sum_{j=1}^{2}\frac{\sigma_{ij}}{z-a_{ij}(t)}\frac{d a_{ij}}{dt}
    +\sum_{k=1}^{n-1}\frac{\sigma_{k}}{z-a_{k}(t)}\frac{d a_{k}}{dt}=
\]
\[
    =\left(\sum_{l=1}^{m}\frac{1}{z-\lambda_{l}(t)}+\sum_{i=1}^{m}\sum_{j=1}^{2}\frac{\sigma_{ij}}{z-a_{ij}(t)}+\sum_{k=1}^{n-1}\frac{\sigma_{k}}{z-a_{k}(t)}\right)z\left(z-1\right)\sum_{i=1}^{m}C_{i}(t)\frac{\lambda_{i}(t)-1}{\lambda_{i}(t)-z}
\]
\begin{equation*}
    +\sum_{i=1}^{m}C_{i}(t)\frac{\lambda_{i}(t)\left(\lambda_{i}(t)-1\right)^{2}}{\left(\lambda_{i}(t)-z\right)^{2}}-\sum_{i=1}^{m}C_{i}(t)\left(\lambda_{i}(t)-1\right).
\end{equation*}
Comparing the residues of the left- and right-hand parts of the last relation at the points $\lambda_{j}(t), a_{ij}(t), a_{l}(t)$ and the free terms,  we obtain the system (\ref{2.17})--(\ref{2.20}).
\end{proof}

We note that the converse statement is also true. If the accessory parameters satisfy the system (\ref{2.17})--(\ref{2.20}), then the family of mappings $f(z,t)$ satisfies (\ref{2.16}). This implies that the solution to the ODE is unique, since the family of conformal mappings given by the Loewner equation is uniquely determined.

Integration of equations (\ref{2.17})--(\ref{2.20}) gives the dependence of the accessory parameters on the parameter $t$.

Let us now describe how to define accessory parameters for a given polygon $P$. First, one should fix the polygon $P_0$ inside of which slits grow. Its specific form is chosen from considerations of convenience or the specifics of the problem, while all accessory parameters for $P_0$ must be known. Next, the Cauchy problem is solved, in which the initial conditions (for $t=0$) are the values of the accessory parameters for the polygon $P_0$. The values of the parameters $a_{i}$, $a_{ij}$, $\lambda_{i}$ at $t=T$ give the desired values for the polygon $P$.

If the shape of $P$ is rather complicated and it is not possible to cut it out of $P_0$ in one step, the above procedure is repeated. At subsequent stages, the initial conditions for the Cauchy problem are determined by the polygonal domain obtained at the previous step. Since the domains remains polygonal after cutting off, this procedure can be carried out sequentially any finite number of times.

Note that when integrating the system of equations (\ref{2.17})--(\ref{2.20}), there is a difficulty associated with the presence of degeneracy at the initial and, possibly, at the final moments.

As $t\to 0$, three the accessory parameters corresponding to the end of the slit and two vertices at its base, lying on different sides of the slit merge into one. When solving the system numerically, one can get rid of the first problem by slightly changing the initial data, for example, by taking $a_{ij}\left(0\right)=\lambda_i\left(0\right)+(-1)^{j}\epsilon_{ij}$, $1\leq i\leq m$, $j=1$, $2$ where $\epsilon_{ij}>0$ is sufficiently small\footnote{In the numerical examples, given below, we took $\epsilon_{ij}=10^{-15}$.}, and thus removing the degeneracy. However, with a rigorous justification of this method, the question arises about the stability of the system with respect to changes in the initial conditions.

As $t\to T$, a degeneracy may appear, since the end of the slit $\Lambda_{k}(t)$ approaches some point of the boundary $E_{k}$. % and their inverse images $\lambda_{k}(t)$ and $e_{k}(t)$ also approach.
As we mentioned in Introduction, the polygonal trajectory of the slit subdivide the polygonal domain $P'$ into two subdomains, $P$ and another one, denoted it by $P''$. According to the G.~D.~Suvorov convergent theorems \cite{suvorov}, the preimages of all the vertices lying on the boundary of $P''$ and the preimage of the endpoint of the slit converge to the same point on $\partial \mathbb{H}^+$, as $t\to T$. This fact may cause complications in numerical calculations.

%\newpage
%%%%%%%%%%%%%%%%%%%%%%%%%%%
%%%%%%%%%%%%%%%%%%%%%%%%%%%
%%%%%%%%%%%%%%%%%%%%%%%%%%%
\section{Solution of the system with the help of  power series}
%%%%%%%%%%%%%%%%%%%%%%%%%%%
%%%%%%%%%%%%%%%%%%%%%%%%%%%
%%%%%%%%%%%%%%%%%%%%%%%%%%%

After we have derived the system of differential equations for determining the accessory parameters, the question arises of how to solve it.

Consider a special case. Assume that $C_i(t)$ are some given real-analytic functions of the variable $x=\sqrt{t}$, $t\in [0,T]$. Then they can be represented by power series
\begin{equation*}
    C_{p}(t)=\widetilde{C}_{p}\left(x\right)=\sum_{n=0}^{\infty}\widetilde{C}_{p,n}x^{n},\ 1\leq p\leq m.
\end{equation*}

Similarly to the case of one slit, we use the change of variable $x=\sqrt{t}$ in  the system (\ref{2.17})--(\ref{2.20}) and seek for its solution in the form of power series:
\begin{equation}
    \lambda_{p}\left(x\right)=\sum_{n=0}^{\infty}\lambda_{p,n}x^{n}, \quad 1\le p \le m,
    \label{rad1}
\end{equation}
\begin{equation}
   a_{k}\left(x\right)=\sum_{m=0}^{\infty}a_{k,m}x^{m},\quad  1\le l\le n-3,
   \label{rad2}
\end{equation}
\begin{equation}
    a_{ij}\left(x\right)=\lambda_{i,0}+\sum_{m=1}^{\infty}a_{ij,m}x^{m},\quad 1\le  i\le m,\ j=1,2.
    \label{rad3}
\end{equation}
Substituting these expansions into the system (\ref{2.17})--(\ref{2.20}) and equating the coefficients at the first power of $x$ in the left- and right-hand parts, we obtain:
\begin{equation*}
    \lambda_{p,1}=q_{p}\sum_{j=1}^{2}\frac{\sigma_{pj}}{\lambda_{p,1}-a_{pj,1}}\,, \quad 1\le p \le m,
\end{equation*}
\begin{equation*}
    a_{ij,1}=\frac{q_{i}}{a_{ij,1}-\lambda_{i,1}}\,,\quad 1\le  i\le m,\ j=1,2,
\end{equation*}
\begin{equation*}
    a_{l,1}=0,\quad  1\le l\le n-3,
\end{equation*}
where $q_{\mu }=-2\widetilde{C}_{\mu ,0}\lambda_{\mu ,0}\left(\lambda_{\mu,0}-1\right)^{2}$. From these equalities it follows that
\begin{equation*}
    \begin{cases}
   \displaystyle \lambda_{i,1}=q_{i}\sum_{j=1}^{2}\frac{\sigma_{ij}}{\lambda_{i,1}-a_{ij,1}}\,,\\[2mm]
   \displaystyle    a_{ij,1}=-\frac{q_{i}}{\lambda_{i,1}-a_{ij,1}}\,,
    \end{cases}\Longrightarrow \ \ \begin{cases}
      \displaystyle \lambda_{i,1}=\left(\alpha_{i1}-\alpha_{i2}\right)\sqrt{\frac{q_{i}}{\alpha_{i1}\alpha_{i2}}}\,, \\[2mm]
      \displaystyle a_{i1,1}=-\sqrt{q_{i}\frac{\alpha_{i2}}{\alpha_{i1}}}\,, \\[3mm]
      \displaystyle a_{i2,1}=\sqrt{q_{i}\frac{\alpha_{i1}}{\alpha_{i2}}}\,.
    \end{cases}
\end{equation*}
Similarly, comparing the coefficients at $x^s$, we obtain
\begin{equation*}
    a_{i1,1}\lambda_{i,s}+\left(sa_{i2,1}-a_{i1,1}\right)a_{i1,s}=\phi_{i1},
\end{equation*}
\begin{equation*}
    a_{i2,1}\lambda_{i,s}+\left(sa_{i1,1}-a_{i2,1}\right)a_{i2,s}=\phi_{i2},
\end{equation*}
\begin{equation*}
    \lambda_{i,s}+\sum_{k=1}^{n-3}\sigma_{k}a_{k,s}+\sum_{j=1}^{2}\sigma_{ij}a_{ij,s}=\phi_{i}^{*},\ 1\leq i\leq m,
\end{equation*}
\begin{equation*}
    a_{k,s}=\phi_{k},\ 1 \leq k \leq n-3,
\end{equation*}
where $\phi_k,\ \phi_{ij}$ and $\phi_i^*$ only depend  on the coefficients of the preceding orders.

Further, arguing similarly to the case of one slit (see \cite{Alexandrov} and \cite{Gutlyansky}), we obtain that the power series (\ref{rad1})--(\ref{rad3}) converge and represent the only solution that is real-analytic with respect to the variable $x=\sqrt{t}$ in some neighborhood of the point $0$.

However, as we will show below, another case is more important for our purposes. It is interesting to study the case where the functions $C_i(t)$ are chosen not arbitrarily, but in such a way that  provides a given dynamics of the slit lengths (up to a parametrization). In this regard, we will discuss this case in the next section.

%%%%%%%%%%%%%%%%%%%%%%%%%%%
%%%%%%%%%%%%%%%%%%%%%%%%%%%
%%%%%%%%%%%%%%%%%%%%%%%%%%%
\section{The control of slit lengths}
%%%%%%%%%%%%%%%%%%%%%%%%%%%
%%%%%%%%%%%%%%%%%%%%%%%%%%%
%%%%%%%%%%%%%%%%%%%%%%%%%%%

Now we will study  how the lengths of the slits depend on the parameter $t$. Consider the dynamics of the endpoint $\Lambda_{r}(t)$ of the $r$-th slit that is the image of the point $\lambda_r(t)$ under the mapping $f(z,t)$, i.e. $\Lambda_{r}(t)=f(\lambda_r(t),t)$. Differentiating this equality, with the help of \eqref{2.3}, we can obtain an expression for the growth rate $v_{r}(t)$  of the length of the $r$-th slit:
\begin{equation}\label{rate}
    v_{r}(t)=\biggl|\frac{d \Lambda_{r}(t)}{dt}\biggr|=|c(t)|A_r(t)C_{r}(t),
\end{equation}
where
\begin{multline}
    A_r(t)=\prod_{i=1}^{m}\prod_{j=1}^{2}|\lambda_{r}(t)-a_{ij}(t)|
    ^{\sigma_{ij}}\prod_{k=1}^{n-1}|\lambda_{r}(t)-a_{k}(t)|^{\sigma_{k}}
\\
       \times\, |\lambda_{r}(t)|\,|\lambda_{r}(t)-1|^{2}\prod_{i=1,i\neq r}^{m}|\lambda_{r}(t)-\lambda_{i}(t)|\,.
    \label{Art}
\end{multline}
We see that $v_{r}(t)$ is proportional to $C_r(t)$ and does not depend explicitly on the other functions $C_j(t)$, $j\neq r$. This fact can be used to control the ratio between the slit lengths.

\subsection{The case of two slits}

\hfill\\Consider the case $m=2$, i.e. when there are only two slits. Assume, for simplicity, that the ratio of the velocities of their lengths is a given positive number $\alpha$, i.~e.
\begin{equation}\label{alph}
\frac{v_{1}(t)}{v_{2}(t)}\equiv \alpha,
\end{equation}
where $v_{1}(t)$ and $v_{2}(t)$ are the growth rates of the lengths of the first and second slits.
Then from \eqref{rate} and \eqref{alph} we obtain
\begin{equation*}
    \frac{C_{1}(t)}{C_{2}(t)}=\alpha\,\frac{A_{2}(t)}{A_{1}(t)}\,.
\end{equation*}
Moreover, $C_{1}(t)+C_{2}(t)=1$, therefore,
\begin{equation}\label{C12}
C_1(t)=\frac{\alpha A_2(t)}{A_1(t)+\alpha A_2(t)}\,, \quad C_2(t)=\frac{A_1(t)}{A_1(t)+\alpha A_2(t)}\,.
\end{equation}
Here $A_1(t)$ and $A_2(t)$ are given by \eqref{Art} where we put $m=2$. If we substitute the obtained expressions for $C_1(t)$ and $C_2(t)$ into the right-hand sides of system \eqref{2.17}--\eqref{2.20}, then we obtain the system to determine the accessory parameters of conformal mappings $f(z,t)$  with the needed ratio $\alpha$ of the growth rates of the lengths of the slits.

Now we will give some examples.\medskip

\begin{ex} Consider a family of conformal mappings $f(z,t)$ from the upper half-plane onto the upper half-plane with two rectilinear slits going orthogonally upwards from given points $x_1$ and $x_{2}$, $x_{1}<x_{2}<0$; for any $t$ (Fig.~\ref{Dva}).

\begin{figure}[h]
	\centering
	\includegraphics[scale=0.3]{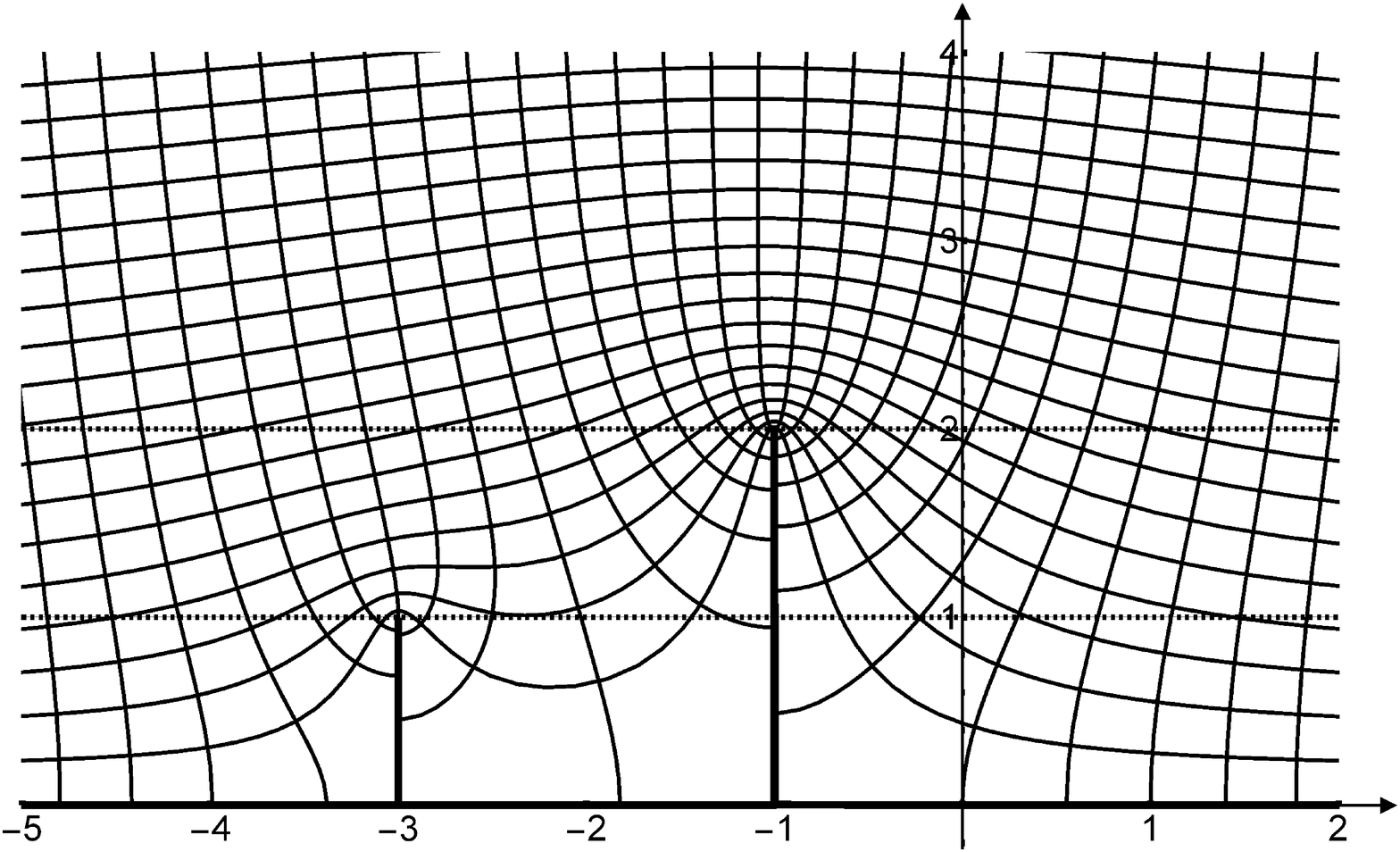}
	\caption{The image of the Cartesian grid under the conformal mapping of the upper half-plane onto a half-plane with two slits.}
		\label{Dva}
\end{figure}

The conformal mapping $f(z,t)$ has the form
\begin{equation}
    f\left(z,t\right)=c(t)\int_{0}^{z}\prod_{l=1}^{2}\left(\zeta-\lambda_{l}(t)\right)\prod_{i,j=1}^{2}\left(\zeta-a_{ij}(t)\right)^{-\frac{1}{2}}d\zeta,
    \label{conftwo}
\end{equation}
where the branch of the integrand is fixed such that it is positive for large positive real $\zeta$;  moreover, $f(z,t)$ is equal to the identity mapping at $t=0$.
From (\ref{conftwo}) it follows that the growth rates of the lengths $L_j=L_j(t)$ of the slits are
\begin{equation*}
    v_{1}(t)=\frac{d  L_{1}(t)}{dt}=\frac{C_{1}(t)\left|c(t)\right|\lambda_{1}(t)\left(\lambda_{1}(t)-1\right)^{2}\left(\lambda_{1}(t)-\lambda_{2}(t)\right)}{\sqrt{\left(\lambda_{1}(t)-a_{11}(t)\right)
    \left(a_{12}(t)-\lambda_{1}(t)\right)\left(a_{21}(t)-\lambda_{1}(t)\right)\left(a_{22}(t)-\lambda_{1}(t)\right)}}\,,
\end{equation*}
\begin{equation*}
    v_{2}(t)=\frac{d L_{2}(t)}{dt}=\frac{C_{2}(t)\left|c(t)\right|\lambda_{2}(t)\left(\lambda_{2}(t)-1\right)^{2}\left(\lambda_{1}(t)-\lambda_{2}(t)\right)}{\sqrt{\left(\lambda_{2}(t)-a_{11}(t)\right)
    \left(\lambda_{2}(t)-a_{12}(t)\right)\left(\lambda_{2}(t)-a_{21}(t)\right)\left(a_{22}(t)-\lambda_{2}(t)\right)}}\,.
\end{equation*}
Taking into account the ratio of growth rates of the slit lengths \eqref{alph}, we get
\[
    \frac{1}{C_{2}(t)}=\alpha\sqrt{\frac{\left(\lambda_{1}(t)-a_{11}(t)\right)\left(a_{12}(t)-\lambda_{1}(t)\right)\left(a_{21}(t)
    -\lambda_{1}(t)\right)\left(a_{22}(t)-\lambda_{1}(t)\right)}{\left(\lambda_{2}(t)-a_{11}(t)\right)\left(\lambda_{2}(t)-a_{12}(t)\right)\left(\lambda_{2}(t)-a_{21}(t)\right)\left(a_{22}(t)-\lambda_{2}(t)\right)}}
\]
\begin{equation*}
    \times\,\frac{\lambda_{2}(t)}{\lambda_{1}(t)}\left(\frac{\lambda_{2}(t)-1}{\lambda_{1}(t)-1}\right)^{2}+1,\qquad C_{1}(t)=1-C_{2}(t),
\end{equation*}
what is equivalent to \eqref{C12}.

Now we will give the results of some numerical calculations. Let $x_1=-2$, $x_{2}=-1$ and we need to find the conformal mapping of the upper half-plane onto the half-plane with two slits orthogonal to the boundary of lengths $L_1 = 1$ and $L_2 = 2$ (Fig. \ref{Dva}). We put $\alpha=L_1/L_2=0.5$. Solving the Cauchy problem for the system of ODEs we obtain the numerical values of the accessory parameters in \eqref{conftwo}. They are given in Table~\ref{2.P-P2}.  The calculations were carried out using the Wolfram Mathematica package. For comparison, we also give in the table the values calculated with the help the well-known Driscoll's SC Toolbox package for MATLAB \cite{Driscoll}; the values are given with $7$ digits after the decimal point. As you can see, the obtained values of the  parameters coincides up to 6 digits after the decimal point.

\begin{table}[h]\caption{The values of the accessory parameters for the conformal mapping of the upper half plane onto the half-plane with two slits.} \label{2.P-P2}
\begin{center}
\begin{tabular}{|c|c|c|}
    \hline
    \raisebox{-0.1cm}{Parameter} & \raisebox{-0.1cm}{Our method} & \raisebox{-0.1cm}{SC Toolbox} \\[1.5ex]
    \hline
      \raisebox{-0.1cm}{$c$} &
      \raisebox{-0.1cm}{$0.5867804$} &
      \raisebox{-0.1cm}{$-$}\\ [1.5ex]
      \hline
      \raisebox{-0.1cm}{$a_{11}$} &
      \raisebox{-0.1cm}{$-9.8974995$} &
      \raisebox{-0.1cm}{$-9.8974994$}\\ [1.5ex]
      \hline
      \raisebox{-0.1cm}{$\lambda_{1}$} &
      \raisebox{-0.1cm}{$-8.5126732$} &
      \raisebox{-0.1cm}{$-8.5126732$}\\ [1.5ex]
      \hline
      \raisebox{-0.1cm}{$a_{12}$} &
      \raisebox{-0.1cm}{$-7.3979258$} &
      \raisebox{-0.1cm}{$-7.3979257$}\\ [1.5ex]
      \hline
      \raisebox{-0.1cm}{$a_{21}$} &
      \raisebox{-0.1cm}{$-6.8108252$} &
      \raisebox{-0.1cm}{$-6.8108251$}\\ [1.5ex]
      \hline
      \raisebox{-0.1cm}{$\lambda_{2}$} &
      \raisebox{-0.1cm}{$-3.7393888$} &
      \raisebox{-0.1cm}{$-3.7393887$}\\ [1.5ex]
      \hline
      \raisebox{-0.1cm}{$a_{22}$} &
      \raisebox{-0.1cm}{$-0.3978735$} &
      \raisebox{-0.1cm}{$-0.3978735$}\\ [1.5ex]
      \hline
    \end{tabular}
\end{center}
\end{table}

To eliminate the degeneracy of the system at the initial moment, a slight change in the initial data was made: instead of $a_{ij}\left(0\right)=\lambda_i\left(0\right)$ we took $a_{ij}\left(0\right)=\lambda_i\left(0\right)+(-1)^{j}\epsilon_{ij}$, $i$, $j=1$, $2$, $\epsilon_{ij}=10^{-15}$.\medskip
\end{ex}

\begin{ex}
Consider the problem of finding the accessory parameters for the conformal mapping of the upper half-plane onto the hexagon $H$ that is  the rectangle $R=[-1,1]\times [0,1]$ with the removed corner $[-1,-0.5]\times [0.5,1]$ (Fig.~\ref{Kirpich}).

As an initial polygon, we take the rectangle $R$. The initial mapping of $\mathbb{H}^+$ onto $R$ is
\begin{equation}\label{fex2}
    f\left(z,0\right)=-c\left(0\right)\int_{0}^{z}\left(\zeta\left(\zeta-1\right)\right)^{-\frac{1}{2}}\left(\zeta-a_{2}\left(0\right)\right)^{-\frac{1}{2}}\left(\zeta-a_{1}\left(0\right)\right)^{-\frac{1}{2}}d\zeta+1,
\end{equation}
where
$$
c\left(0\right)=\frac{2+\sqrt{2}}{K\left({1}/{\sqrt2}\right)}\,=1.84146496\ldots,\quad a_{1}\left(0\right)=-(3+2\sqrt{2})=-5.82842712\ldots,$$
$$ a_{2}\left(0\right)=-{2}(\sqrt{2}+1)=-4.82842712\ldots
$$
Here $K\left(k\right)$ is the complete elliptic integral of the first kind.  The branch of the integrand is fixed so that it is positive for large positive real $\zeta$; the same applies to the mappings $f(z,t)$ given below.

Let two slits go, orthogonally to the boundary of the rectangle $R$, from the points $-0.5+i$ and $-1+0.5i$ with the same growth rates.
We note that the preimages of the points $-0.5+i$ and $-1+0.5i$ under the mapping \eqref{fex2} are
$$
\lambda_1= -(\sqrt{2+\sqrt{2}}+1)(1+\sqrt{2})=-6.87509856\ldots, \quad \lambda_2=-(1+\sqrt[4]{2})(1+\sqrt{2})=-5.28521351\ldots.
$$

The Schwarz-Christoffel integral, mapping $\mathbb{H}^{+}$ onto $R$ with two slits, described above, and normalized by the conditions $f\left(0,t\right)=1, f\left(1,t\right)=1+i, f\left(\infty,t\right)=i$, has the form
\begin{equation}
    f\left(z,t\right)=-c(t)\int_{0}^{z}(\zeta(\zeta-1))^{-\frac{1}{2}}\prod_{l=1}^{2}\left(\zeta-\lambda_{l}(t)\right)
    \prod_{i,j=1}^{2}\left(\zeta-a_{ij}(t)\right)^{-\frac{1}{2}}\prod_{k=1}^{2}\left(\zeta-a_{k}(t)\right)^{-\frac{1}{2}}d\zeta+1,
\end{equation}
where $-\infty<a_{11}(t)<\lambda_1(t)<a_{12}(t)<a_1(t)<a_{21}(t)<\lambda_2(t)<a_{22}(t)<a_2(t)<0$. We note that the image of $\infty$ is not an angular point but it does not matter, according to Remark~\ref{r2}. For $t=0$, we have $a_{11}(0)=\lambda_1(0)=a_{12}(0)=\lambda_1$, $a_{21}(0)=\lambda_2(0)=a_{22}(0)=\lambda_2$.

%We will assume that the lengths of the slits coincide and equal $t$, $0\le t\le 0.5$.
At the chosen growth rates, the slits simultaneously reach the point $0.5+0.5i$ for some $t=T$, and the family of mappings $f(z,t)$ converges to the mapping $g\left(z\right)$ of $\mathbb{H}^+$ onto the needed hexagon $H$; it has the form
$$
f(z,T)=-c\int_{0}^{z}(\zeta(\zeta-1)(\zeta-a))^{-\frac{1}{2}}(\zeta-\lambda)^{1/2}\prod_{i=1}^{2}(\zeta-b_{i})^{-\frac{1}{2}}d\zeta+1,
$$
$-\infty<b_1<\lambda<b_2<a<0$. Moreover, $a_{11}(t)\to b_1$; $\lambda_1(t)$, $a_{12}(t)$, $a_1(t)$, $\lambda_2(t)$ and $a_{21}(t)\to \lambda$,  $a_{22}(t)\to b_2$,  $a_2(t)\to a$, as $t\to T$.
Table~\ref{2.Kirpich} compares the results of our method and those obtained by the SC Toolbox package; the values are given with $7$ digits after the decimal point. As we can see, most of the obtained values of the parameters agree up $6$ digits after the decimal point; the values of $\lambda$ agree up $4$ digits; this is caused by the fact that the convergence of the mappings is non-uniform near this point. %, and in the vicinity of the point of contact of the slits, the accuracy decreases.}

\begin{table}[h]\caption{The values of the accessory parameters for the conformal mapping of the upper half plane $\mathbb{H}^+$ onto the hexagon $H$.} \label{2.Kirpich}
\begin{center}
\begin{tabular}{|c|c|c|}
    \hline
    \raisebox{-0.1cm}{Parameter} &
    \raisebox{-0.1cm}{Our method} &
    \raisebox{-0.1cm}{SC Toolbox} \\[1.5ex]
    \hline
      \raisebox{-0.1cm}{$c$} &
      \raisebox{-0.1cm}{$1.90896$} &
      \raisebox{-0.1cm}{$-$}\\ [1.5ex]
      \hline
      \raisebox{-0.1cm}{$b_1$} &
      \raisebox{-0.1cm}{$-8.6039921$} &
      \raisebox{-0.1cm}{$-8.6039920$}\\ [1.5ex]
      \hline
      \raisebox{-0.1cm}{$\lambda$} &
      \raisebox{-0.1cm}{$-4.6992541$} &
      \raisebox{-0.1cm}{$-4.6992949$}\\ [1.5ex]
      \hline
      \raisebox{-0.1cm}{$b_2$} &
      \raisebox{-0.1cm}{$-4.0805629$} &
      \raisebox{-0.1cm}{$-4.0805627$}\\ [1.5ex]
      \hline
      \raisebox{-0.1cm}{$a$} &
      \raisebox{-0.1cm}{$-4.0225626$} &
      \raisebox{-0.1cm}{$-4.0225623$}\\ [1.5ex]
      \hline
    \end{tabular}
\end{center}
\end{table}

\begin{figure}[h]
	\centering
	\includegraphics[scale=0.3]{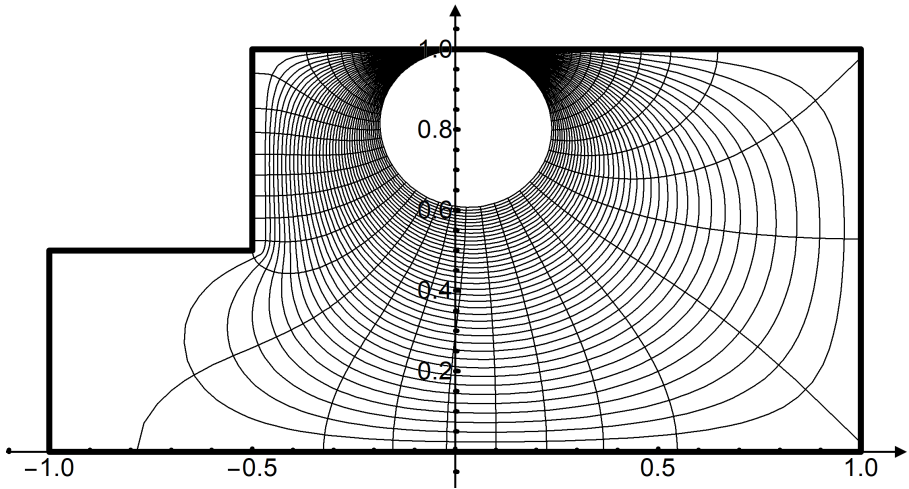}
	\caption{The image of the Cartesian grid under the conformal mapping of the half-plane $\mathbb{H}^+$ onto the hexagon $H$.}
	\label{Kirpich}
\end{figure}
\end{ex}

\newpage \subsection{General case}\label{Gc}
\hfill\\ In general, by setting the relationships between the velocities $v_{j}$ of the lengths of the slits, which could be even non-stationary, we have $m-1$ equations for $m$ functions $C_{i}(t)$:
\begin{equation*}
    g_{\nu}\left(v_{1}(t),v_{2}(t)\ldots, v_{m}(t),t\right)=0,\ 1\leq\nu\leq m-1.
\end{equation*}
The condition
\begin{equation}
    \sum_{i=1}^{m}C_{j}(t)=1
    \label{S2}
\end{equation}
closes the system of equations to determine $C_{i}(t)$, $1\le i \le m$. Together with the equations (\ref{2.17})-(\ref{2.20}) they form a complete system to find the accessory parameters.

Now consider the very important case where the ratios of the rates $v_{j}$ are constant. Let, for example,
$$
\frac{v_j(t)}{v_m(t)}\equiv \alpha_j,\quad 1\le j\le m-1,
$$
where $\alpha_j$ are some positive constants. Taking into account  \eqref{rate} and  \eqref{S2} we obtain
$$
C_j^{-1}(t)=\alpha_j^{-1}A_j(t)\sum_{k=1}^m\alpha_kA_k^{-1}(t),\quad 1\le j\le m,
$$
where we put $\alpha_m=1$; the values of $A_j(t)$ is given by \eqref{Art}.

Therefore we find $C_j(t)$ as functions depending explicitly on the accessory parameters. Putting them in (\ref{2.17})-(\ref{2.20}) we obtain a system of ODEs to determine the accessory parameters.

%%%%%%%%%%%%%%%%%%%%%%%%%%%
%%%%%%%%%%%%%%%%%%%%%%%%%%%
%%%%%%%%%%%%%%%%%%%%%%%%%%%
\section{Relationship between accessory parameters and side lengths}
%%%%%%%%%%%%%%%%%%%%%%%%%%%
%%%%%%%%%%%%%%%%%%%%%%%%%%%
%%%%%%%%%%%%%%%%%%%%%%%%%%%

The search of a suitable parameter with respect to which the family of mappings will be continuously differentiable is an important and interesting problem. In this section we consider families corresponding to deformations of a polygon with fixed angles. As the above parameter, it is proposed to take any varying side length (in our case, the length of one of the slits). This choice is based on the idea of a bijective and differentiable mapping between the prevertices and the lengths of polygon sides (with fixed angles). The proof of this fact allows us to justify the existence of the derivative $\frac{\partial f}{\partial t}$ of the mapping $f(z,t)$, as well as the existence of the functions  $C_i(t).$

\subsection{Differentiability of polygon side lengths as a function of accessory parameters}
\hfill\\ Consider the set of polygons $P$ satisfying the conditions:

(1) the boundary of $P$ is a closed polygonal line $A_1A_2\ldots A_nA_{n+1}$, where $A_{n+1}=A_1$, and the value of $A_{n-2}$ is fixed;

(2) the angle of inclination $\beta$ of the segment $A_1A_2$  to the real axis is fixed;

(3) the values of interior angles of  $P$, $\alpha_1\pi, \alpha_2\pi,\ldots \alpha_n\pi $ at the vertices $A_1A_2\ldots A_n$ are fixed.

Then the conformal mapping of the upper half-plane onto $P$ is given by the formula
\begin{equation}
    f\left(z\right)=d e^{i\beta}\int_{0}^{z}\prod_{k=1}^{n-1}\left(\zeta-a_{k}\right)^{\alpha_{k}-1}d\zeta+A_{n-2};
    \label{de}
\end{equation}
Here, as above, we assume that $a_{n-2}=0$, $a_{n-1}=1$ and the preimage of $A_n$ is $a_n=\infty$.

The length of the $\nu$-th side of the polygon is:
\begin{equation}\label{lnu}
    l_{\nu}=d \int_{a_{\nu}}^{a_{\nu+1}}\prod_{k=1}^{n}\left|x-a_{k}\right|^{\sigma_{k}}dx.
\end{equation}
We note that if we fix the values of $l_{\nu}$ with $\ 1\leq\nu\leq n-2$, then the values $l_{n-1}$ are  $l_{n}$ are uniquely determined, because of (2). Consider the mapping $\Phi:(d,a_1,a_2,\ldots a_{n-3})\mapsto (l_1,l_2,\ldots,l_{n-2})$ where $d>0$, $a_1<a_2<a_3<\ldots <a_{n-3}<0$ where $l_{\nu}$ is given by \eqref{lnu}.  We note that if some of $\alpha_k-1$ is positive, the mapping  \eqref{de} may be non-univalent and the corresponding image is a multi-sheeted Riemann surface.

\begin{lem}\label{smo}
The mapping $\Phi$ is smooth and non-degenerate at every point of the set $$\mathcal{A}=\{(d, a_1,a_2,a_3,\ldots,a_{n-3})\in\mathbb{R}^{n-2}\mid d>0, a_1<a_2<a_3<\ldots<a_{n-3}<0\}.$$
\end{lem}

\begin{rem}\label{r3} We note that the relationship between the accessory parameters and the side lengths was considered in the work of A.~Vainshtein \cite{Weinstein1}; the result obtained was used to substantiate the continuity method he developed. Variations of the this idea were also used by other scientists: K.~Reppe \cite{Reppe}, L.~N.~Trefesen \cite{Trefethen} and others. In particular, they form the basis for the numerical mapping of the Schwarz-Christoffel integral~\cite{Driscoll_Tr}. Below we will give a sketch of the proof of this fact based on the ideas given in \cite[ch.~3]{Monahov}.\medskip
\end{rem}

Differentiating $l_{\nu}$ in \eqref{lnu} with respect to $a_{\mu}$ and $d$, we can easily show that exist the mapping $\Phi$ is continuously differentiable. It remains to prove that the Jacobian of the mapping $\Phi$ is non-degenerate.
Assume the contrary. Then at some point of $\mathcal{A}$ there exist variations $\delta d$, $\delta a_1,\ldots,\delta a_{n-2}$ under with the values of all $l_\nu$ do not change. If the straight line containing the segment $A_\nu A_{\nu+1}$ is described as $\{(\xi,\eta)\in \mathbb{R}^2: \beta_\nu \xi+\gamma_\nu \eta=\epsilon_\nu\}$, where $\beta_\nu$, $\gamma_\nu$, $\epsilon_\nu$ are some constants, then the variation  $\delta f$ satisfies the boundary condition of a homogeneous Hilbert boundary value problem with piecewise-constant coefficients (see, e.g. \cite{Gahov} or \cite{Monahov}):
$$\beta_\nu \Re \delta f(x)+\gamma_\nu \Im \delta f(x)=0,\quad  a_\nu<x<a_{\nu+1},$$ for all $\nu$. Studying the behavior of $\delta f$ at the points $a_\nu$ we conclude that the Hilbert problem has only a trivial solution. Lemma~\ref{smo} is proved.\medskip

From Lemma~\ref{smo} it follows that the accessory parameters in \eqref{2.3} depends smoothly on the lengths of the  slits.

%This question is reduced to investigation of the variation $\delta f$ of the Schwarz-Christoffel integral under the condition that the variations of all $l_{\nu}$ are zero. Then $\delta f$ is a solution %of a homogeneous Hilbert boundary value problem (see, e.g. \cite{Gahov} or \cite{Monahov}) and, similarly to \cite[ch.3]{Monahov}, where a more general integral representations was studied, we show that %$\delta f\equiv0$, which means that the mapping is non-degenerate.

\subsection{Differentiability of the family of mappings}
\hfill\\ Now consider a family of polygons with are obtained from a fixed  polygon by drawing  $m$ slits of lengths $l_1,l_2\ldots,l_m$. We will assume that  these lengths  depends smoothly on each other. For definiteness, we will assume that $l_j=\phi_j(l_1)$, $2\le j\le m$, where $\phi_j$ are smooth functions. Then the accessory parameters and the corresponding conformal mappings also depends smoothly on $l_1$. Therefore, we can put, as a parameter $t$ for the family $f(z,t)$, the value of $l_1$ and, consequently, consider $f(z,t)$ as a smooth family.

Then we can take the expression following from Loewner's equation $(\ref{2.16})$ as the definition of the functions $C_k(t)$, since all the quantities in it now are well defined:
\begin{equation*}
    C_{k}(t)=\frac{2\pi i}{\lambda_{k}(t)\left(\lambda_{k}(t)-1\right)^{2}}\underset{z=\lambda_{k}(t)}{\mathrm{res}}\left(\frac{\partial f}{\partial t}\left(\frac{\partial f}{\partial z}\right)^{-1}\right),\ 1\leq k\leq m.
\end{equation*}

Now we will show that, under such the smooth parametrization,  the function $q(t)$, defined by \eqref{q1wz}, is a smooth function of the parameter $t=L_1$.
Since $f(z,T)=f(w(z,t),t)$ and $f(z,T)$ is a Christoffel-Schwarz integral, differentiating by $z$ we have $f'(z,T)=f'(w(z,t),t) w'(z,t)$, therefore,
$$
w'(z,t)=\frac{f'(z,T)}{f'(w(z,t),t)}\,.
$$
From \eqref{2.3} we obtain
$$
f'(w(z,t),t)\sim \Psi(w(z,t)-1)^{\sigma_{n-1}},\quad  {f'(z,T)}\sim  C(z-1)^{\sigma_{n-1}}, \ z\to 1,
$$
where $C$ is a constant and $\Psi$ is a non-zero smooth function of $c(t)$ and the preimages of  the vertices of the polygonal domain $D_n(t)$. Therefore, we can consider that $\Psi=\Psi(t)$ is a smooth function of $t$. Then
$$
w'(1,t)=\lim_{z\to 1}\frac{f'(z,T)}{f'(w(z,t),t)}\,= \frac{C}{\Psi(t)(w'(1,t))^{\sigma_{n-1}}},
$$
consequently,
$$
    q(t)=w'(1,t)=\left(\frac{C}{\Psi(t)}\right)^{1/\alpha_{n-1}}.
$$
This shows that $q(t)$ is smooth.

We also note that the parametrization by the length of one of the slits allows us to avoid some technical difficulties. Let we want to obtain the first slit with length $L_1$ (the lengths of the remaining slits are uniquely determined). Then the system of ODEs (\ref{2.17})--(\ref{2.20}) must be solved on the interval  $\left[0,L_{1}\right]$. In the case of an arbitrary parametrization, it is necessary to look for the final value $T$ of the parameter $t$, at which the slits have desired lengths. Here, this value is determined automatically.

Weinstein's result helped us to prove the differentiability of a family of mappings in the case of several slits. However, the consequences are not limited to this: the justification is also applicable to a more general case considered by I.~A.~Kolesnikov in \cite{Kolesnikov2}.

%%%%%%%%%%%%%%%%%%%%%%%%%%%
%%%%%%%%%%%%%%%%%%%%%%%%%%%
%%%%%%%%%%%%%%%%%%%%%%%%%%%
\section{Conclusion}
%%%%%%%%%%%%%%%%%%%%%%%%%%%
%%%%%%%%%%%%%%%%%%%%%%%%%%%
%%%%%%%%%%%%%%%%%%%%%%%%%%%

We propose a generalization of the modification of the well-known Kufarev's method for approximate finding the conformal mapping of the upper half-plane onto a polygon, given by V.~Ya.~Gutlyanskii and A.~O.~Zaidan. It is based on the Loewner parametric method and considering one-parameter families $f(z,t)$, $0\le t\le T$, of Schwarz-Christoffel integrals; for a fixed $t$, the functions $f(z,t)$ maps the upper half-plane conformally onto a polygon with a few rectilinear slits and the lengths of the slits increases with the growth of $t$.
The integral representation of $f(z,t)$ depends on some unknown (accessory) parameters.
The differential equations that describe the dynamics of the mappings $f(z,t)$ and the corresponding accessory parameters are derived. The existence and uniqueness of a one-parameter family implies the existence and uniqueness of a solution to the resulting system of ODEs. We also investigate the problem of finding a suitable parameter $t$. We show that if the parameter $t$ coincides with the length of some slit, and other lengths are depends on this length smoothly, then the family $f(z,t)$ depends on it smoothly, and the accessory parameters can be found by solving a Cauchy problem for a system of ODEs. Numerical calculations confirms the efficiency of the method and its good accuracy.

However, some questions remained unexplored. For example, the stability of the ODE system with respect to the variation of initial data needs to be justified, as well as the use of the power series method in the general case. Also we would like to achieve a greater accuracy in the numerical solution of the obtained system of ODEs to determine the accessory parameters in the Schwarz-Christoffel integral.

%%%%%%%%%%%%%%%%%%%%%%%%%%%
%%%%%%%%%%%%%%%%%%%%%%%%%%%
%%%%%%%%%%%%%%%%%%%%%%%%%%%
\section{Acknowledgements}
%%%%%%%%%%%%%%%%%%%%%%%%%%%
%%%%%%%%%%%%%%%%%%%%%%%%%%%
%%%%%%%%%%%%%%%%%%%%%%%%%%%

The work of A.~Posadsky and S.~Nasyrov was supported by the Ministry of Science and Higher Education of the Russian Federation (agreement no. 075-15-2022-287).

%%%%%%%%%%%%%%%%%%%%%%%%%%%
%%%%%%%%%%%%%%%%%%%%%%%%%%%
%%%%%%%%%%%%%%%%%%%%%%%%%%%

\end{document}